\documentclass[11pt]{article}%
\usepackage{amssymb}
\usepackage{eurosym}
\usepackage{amsfonts}
\usepackage{amsmath}
\usepackage{graphicx}%
\usepackage{epstopdf}
\providecommand{\U}[1]{\protect\rule{.1in}{.1in}}
\newtheorem{theorem}{Theorem}
\newtheorem{acknowledgement}[theorem]{Acknowledgement}

\newtheorem{definition}[theorem]{Definition}

\newtheorem{lemma}[theorem]{Lemma}

\newtheorem{proposition}[theorem]{Proposition}

\newenvironment{proof}[1][Proof]{\noindent\textbf{#1.} }{\ \rule{0.5em}{0.5em}}
\begin{document}

\title{On the existence of short trajectories of quadratic differentials related to
generalized Jacobi polynomials with non real varying parameters.}
\author{Mondher Chouikhi and Faouzi Thabet\\University of Gab\`{e}s, Tunisia}
\maketitle

\begin{abstract}
The study of the asymptotic distributions of zeros of generalized Jacobi
polynomials with non real varying parameters, leads with quadratic
differentials. In fact, the support of the limit measure of the root-counting
measures sits on the finite critical trajectories of a related quadratic differential.

In this paper, we study the trajectories of this quadratic differential, more
precisely, we give a necessary and sufficient condition on the complex numbers
$a,b,$ and $\lambda$ for the existence of at list one finite critical
trajectory of the quadratic differential $\frac{\lambda^{2}\left(  z-a\right)
\left(  z-b\right)  }{\left(  z^{2}-1\right)  ^{2}}dz^{2}$.

\end{abstract}

\footnote{
Trajectories and orthogonal trajectories of a quadratic differential.
Homotopic Jordan arcs.
\par
Corresponding Author :faouzithabet@yahoo.fr}\vspace{1cm}

\section{Introduction}

This paper is a continuation of oldest works motivated by the large-degree
analysis of the behavior of the Jacobi polynomials $P_{n}^{(\alpha,\beta)}$,
when the parameters $\alpha,\beta\in%
\mathbb{C}
$ depend on the degree $n$ linearly. Recall that these polynomials can be
given explicitly by (see \cite{Szego})
\begin{equation}
P_{n}^{(\alpha,\beta)}\left(  z\right)  =2^{-n}\sum_{k=0}^{n}\left(
\begin{array}
[c]{c}%
n+\alpha\\
n-k
\end{array}
\right)  \left(
\begin{array}
[c]{c}%
n+\beta\\
\ k
\end{array}
\right)  \left(  z-1\right)  ^{k}\left(  z+1\right)  ^{n-k}, \label{jacobiDef}%
\end{equation}
where $\left(
\begin{array}
[c]{c}%
\gamma\\
k
\end{array}
\right)  =\frac{\gamma\left(  \gamma-1\right)  ...\left(  \gamma-k+1\right)
}{k!}$ for $\left(  \gamma,k\right)  \in%
\mathbb{C}
\times%
\mathbb{N}
.$ Equivalently, the Jacobi polynomials can be defined by the well-known
Rodrigues formula
\[
P_{n}^{(\alpha,\beta)}\left(  z\right)  =\frac{1}{2^{n}n!}\left(  z-1\right)
^{-\alpha}\left(  z+1\right)  ^{-\beta}\left(  \frac{d}{dz}\right)
^{n}\left[  \left(  z-1\right)  ^{n+\alpha}\left(  z+1\right)  ^{n+\beta
}\right]  .
\]
Clearly, polynomials $P_{n}^{(\alpha,\beta)}$ are entire functions of the
complex parameters $\alpha,\beta$.

The classical case is when $\alpha,\beta>-1$: for these values of the
parameters, the Jacobi polynomials are orthogonal on $[-1,1]$ with respect to
the weight function $(1-x)^{\alpha}(1+x)^{\beta}$.

We consider the sequence $P_{n}^{nA,nB},$ as $n\rightarrow\infty$. The case
$A,B\geq0$ can be tackled using the standard tools related to varying
orthogonality and equilibrium measures in an external field on $%
\mathbb{R}
$, see e.g.~\cite{Gonchar Rakhmanov}, or \cite{Gawronsky}. The general
situation $A,B\in%
\mathbb{R}
$ was analyzed in\cite{ABJK AMF RO},\cite{ABJKuij AMF},\cite{AMF PGM RO}. The
situation $A\notin%
\mathbb{R}
,$ $B>0$ was analyzed in \cite{AMF FT}. In this paper, we are interested in
the situation when
\begin{equation}
A\notin%
\mathbb{R}
,B\notin%
\mathbb{R}
,A+B+1\neq0,A+B+2\neq0. \label{cond sur Aet B}%
\end{equation}

\begin{definition}
\bigskip For a compactly supported finite complex-valued Borel measure $\mu$,
we define its \emph{Cauchy transform} $\mathcal{C}_{\mu}$ as
\[
\mathcal{C}_{\mu}\left(  z\right)  =\int_{%
\mathbb{C}
}\frac{d\mu\left(  t\right)  }{z-t},\quad z\in%
\mathbb{C}
\setminus supp\left(  \mu\right)  .
\]

\end{definition}

For instance, if $P$ is a polynomial of degree $n,$ then the Cauchy transform
$\mathcal{C}_{P}$ of its \emph{normalized root-counting measure} $\frac{1}%
{n}\sum_{p\left(  a\right)  =0}\delta_{a}$ (where $\delta_{a}$ is the Dirac
measure supported at $a$) is given by
\[
\mathcal{C}_{p}\left(  z\right)  =\frac{1}{n}\sum_{p\left(  a\right)  =0}%
\frac{1}{z-a}=\frac{P^{\prime}\left(  z\right)  }{nP\left(  z\right)  }.
\]
The Cauchy transform of a compactly supported finite complex-valued Borel
measure $\mu$ defines an analytic function in $%
\mathbb{C}
\diagdown supp\left(  \mu\right)  $ satisfying the properties
\[
\mathcal{C}_{\mu}\left(  z\right)  \sim\frac{\mu\left(
\mathbb{C}
\right)  }{z},\quad z\longrightarrow\infty;\quad\mu=\frac{1}{\pi}%
\frac{\partial\mathcal{C}_{\mu}}{\partial\overline{z}}.
\]
The Cauchy transform of a non compactly supported finite complex-valued Borel
measure $\mu$ can defined in the distribution sense.

The following Theorem gives the connection between the limit behavior of the
zeros of the sequence \ref{jacobiDef}, and the structure of trajectories of a
particular quadratic differential.

\begin{theorem}
\label{shapiro+solynin}Suppose that a sequence of Jacobi polynomials
$P_{n}^{(\alpha_{n},\beta_{n})}$ satisfies conditions:

\begin{enumerate}
\item[(i)] $\lim_{n\rightarrow\infty}\frac{\alpha_{n}}{n}=A,\lim
_{n\rightarrow\infty}\frac{\beta_{n}}{n}=B;$ $A,B$ satisfy conditions
\ref{cond sur Aet B}.

\item[(ii)] the sequence $\left\{  \mu_{n}\right\}  $ of the corresponding
root-counting measures converges weakly to a compactly supported probability
measure $\mu$ in $%
\mathbb{C}
$.
\end{enumerate}

Then the Cauchy transform $\mathcal{C}_{\mu}$ satisfies almost everywhere in $%
\mathbb{C}
$ the following quadratic equation:%
\[
(1-z^{2})\mathcal{C}_{\mu}^{2}-((A+B)z+A-B)\mathcal{C}_{\mu}+A+B+1=0.
\]
Moreover, the support of $\mu$ consists of finitely many critical horizontal
trajectories of the quadratic differential%
\[
\varpi_{A,B}=-\frac{R_{A,B}\left(  z\right)  }{\left(  z^{2}-1\right)  ^{2}%
}\,dz^{2},
\]
where
\begin{equation}
R_{A,B}\left(  z\right)  =\left(  A+B+2\right)  ^{2}z^{2}+2\left(  A^{2}%
-B^{2}\right)  z+\left(  A-B\right)  ^{2}-4\left(  A+B+1\right)  .
\label{defRAB}%
\end{equation}

\end{theorem}

\begin{proof}
See e.g \cite{bullgard}, \cite{shapiro}; or \cite{shapiro soly}, and
references therein.
\end{proof}

\bigskip The critical graph $\Gamma_{A,B}$ of the quadratic differential
$\varpi_{A,B},$ with $A,B\in%
\mathbb{R}
$ depends on the sign of
\[
\Delta=\left(  A+1\right)  \left(  B+1\right)  \left(  A+B+1\right)  .
\]
If $\Delta>0,$ then $\Gamma_{A,B}$ is formed by two loops around $-1$ and $1$,
joined by a segment included in $\left(  -1,1\right)  .$ If $\Delta<0,$ then
$\Gamma_{A,B}$ is formed by two loops around $-1$ and $1,$ with common edge
crossing $\left(  -1,1\right)  .$ See Figure \ref{Fig1}.%
\begin{figure}[h]%
\centering
\fbox{\includegraphics[
height=2.0903in,
width=3.7048in
]%
{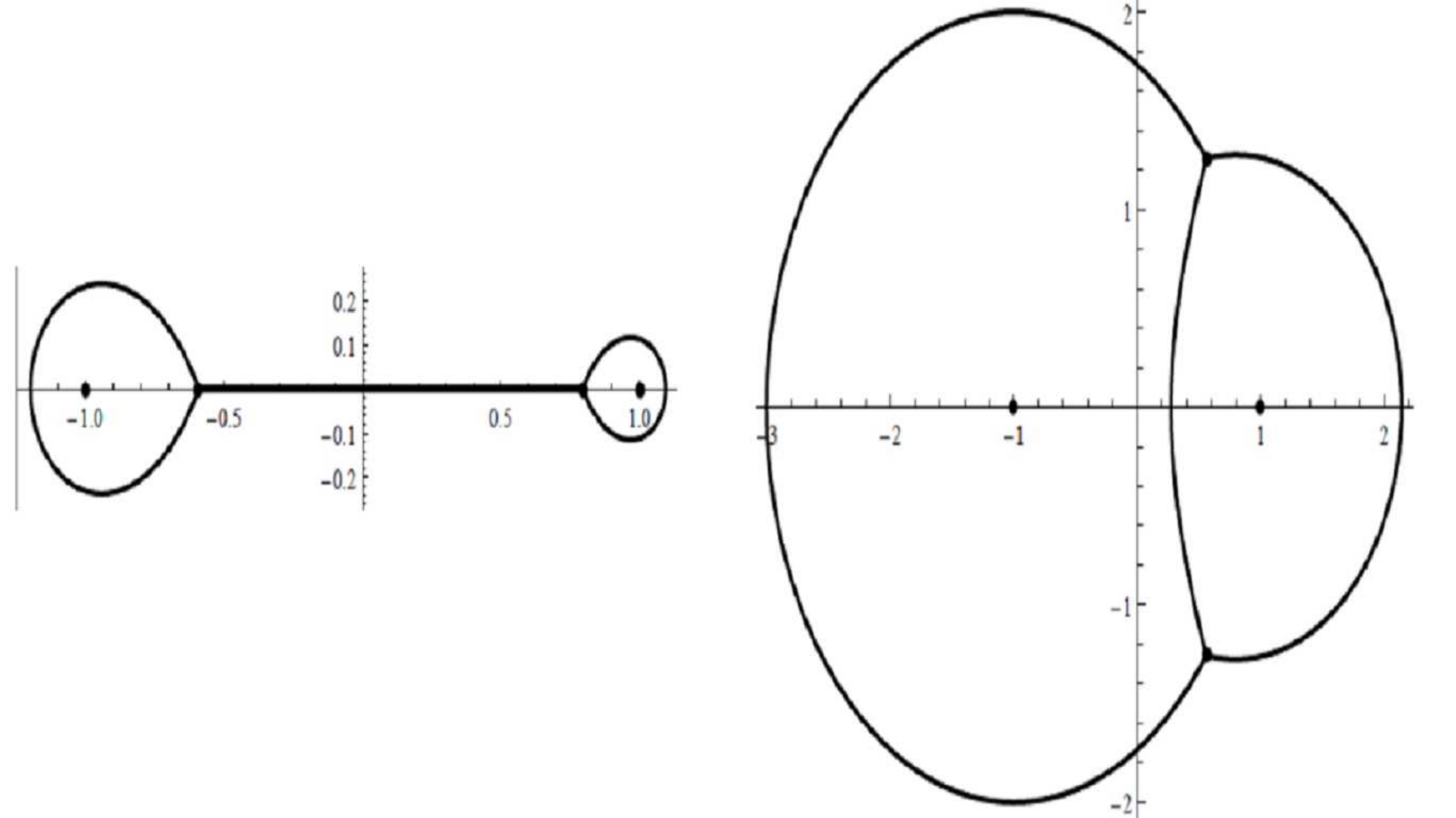}%
}\caption{Critical graph when $A=2,B=3$ (left); $A=-2,B=3$ (right).}%
\label{Fig1}%
\end{figure}

For the case $A\notin%
\mathbb{R}
,B>0,$ $\Gamma_{A,B}$ is formed by a short trajectory, a loop around $-1$, and
two critical trajectories diverging to $1$ and $\infty.$ See Figure \ref{Fig2}%
\begin{figure}[h]%
\centering
\fbox{\includegraphics[
height=2.0903in,
width=3.7048in
]%
{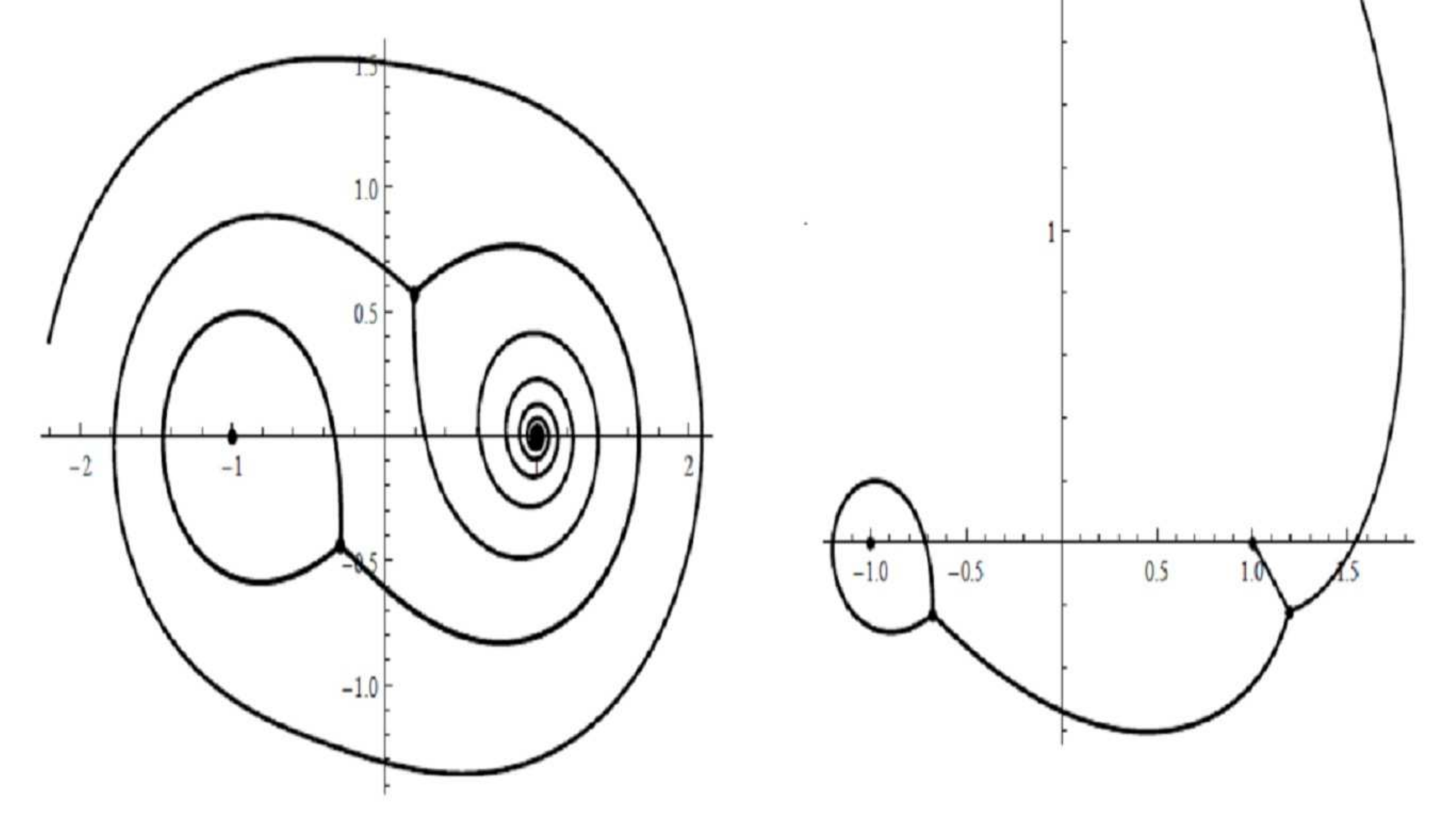}%
}\caption{Critical graph when $A=-1.1+0.1i,$ $B=1$ (left); $A=1.5i,$ $B=2$
(right).}%
\label{Fig2}%
\end{figure}

\bigskip

The main result of this paper is the following theorem

\begin{theorem}
\label{thm:1} Let $A,B$ satisfy assumptions (\ref{cond sur Aet B}). Then, the
structure of the critical graph $\Gamma_{a,b}$ of $\varpi_{A,B}$ is as follows:

\begin{itemize}
\item If $A+B\notin%
\mathbb{R}
,$ then there exist :

\begin{itemize}
\item one short trajectory $\gamma_{A,B}$ of $\varpi_{A,B}$, joining the zeros
$\zeta_{-}$ and $\zeta_{+}$ of $R_{A,B},$

\item two infinite critical trajectories emanating from each zero diverging
differently to $+1,$ $-1,$ or $\infty.$
\end{itemize}

\item If $A+B\in%
\mathbb{R}
,$ then there exist :

\begin{itemize}
\item two short trajectories of $\varpi_{A,B}$, joining $\zeta_{-}$ and
$\zeta_{+},$ and forming a Jordan curve encircling $\pm1,$

\item a trajectory emanating from each zero diverging differently to $+1$ or
$-1.$
\end{itemize}
\end{itemize}
\end{theorem}

It was shown in \cite{AMF FT} that the existence of the short trajectory
$\gamma$ joining the zeros $\zeta_{-}$ and $\zeta_{+}$, is the corner stone in
the study of the asymptotic zero distribution of the Jacobi polynomials; it
satisfies
\[
\int_{\gamma}\frac{\left(  \sqrt{R_{A,B}\left(  t\right)  }\right)  _{+}%
}{t^{2}-1}dt=\pm2\pi i.
\]
The short trajectory $\gamma$ is the support of the measure $\mu$ limit in the
weak-* topology of the sequence $\left\{  \mu_{n}\right\}  $ in Theorem
\ref{shapiro+solynin}. Following the pioneering works of Gonchar-Rakhmanov
\cite{Gonchar Rakhmanov}, and Stahl \cite{Stahl}, the weak asymptotic of the
polynomials measure $\mu$ is absolutely continuous with respect to the linear
Lebesgue measure on $\gamma$ and given by the formula
\[
d\mu\left(  z\right)  =\frac{1}{2\pi}\frac{\left(  \sqrt{R_{A,B}\left(
z\right)  }\right)  _{+}}{z^{2}-1}dz.
\]
The strong uniform asymptotic can be tackled via the Riemann-Hilbert steepest
descent method of Deift-Zhou \cite{Deift}.

\section{The quadratic differential $\frac{\lambda^{2}\left(  z-a\right)
\left(  z-b\right)  }{\left(  z^{2}-1\right)  ^{2}}dz^{2}$}

In this section we focus on the quadratic differential on the Riemann sphere
$\widehat{%
\mathbb{C}
}$ defined by
\[
\varpi_{a,b,\lambda}=\frac{\varphi_{a,b,\lambda}\left(  z\right)  }{\left(
z^{2}-1\right)  ^{2}}dz^{2}=\frac{\lambda^{2}\left(  z-a\right)  \left(
z-b\right)  }{\left(  z^{2}-1\right)  ^{2}}dz^{2},
\]
where $a,b,$ and $\lambda$ are three complex numbers such that
\begin{equation}
a\neq b,a,b\notin\left\{  -1,1\right\}  ,\lambda\neq0. \label{cond a b}%
\end{equation}

\bigskip

\begin{definition}
\label{P}We say that Property $\mathcal{P}_{a,b,\lambda}$ is satisfied, if the
imaginary part of at list one of the following four numbers (see Proposition
\ref{premier} )vanishes%
\begin{equation}
\pm\pi i\lambda\left(  \sqrt{\left(  a-1\right)  \left(  b-1\right)  }\pm
\sqrt{\left(  a+1\right)  \left(  b+1\right)  }-2\right)  . \label{im}%
\end{equation}

\end{definition}

The main result of this section is the following

\begin{proposition}
\label{main prop}Let $a,b,$ and $\lambda$ satisfying (\ref{cond a b}). Then,
the quadratic differential $\varpi_{a,b,\lambda}$ has a short trajectory, if
and only if, Property $\mathcal{P}_{a,b,\lambda}$ is satisfied.
\end{proposition}

The \emph{horizontal trajectories} (or just trajectories) of $\varpi
_{a,b,\lambda}$ are the loci of the equation
\begin{equation}
\Im\int^{z}\frac{\sqrt{\varphi_{a,b,\lambda}\left(  t\right)  }}{t^{2}%
-1}\,dt\equiv\text{\emph{const}},\text{(}\Re\int^{z}\frac{\sqrt{\varphi
_{a,b,\lambda}\left(  t\right)  }}{t^{2}-1}\,dt\text{ \emph{monotonic})}
\label{re}%
\end{equation}
while the \emph{vertical} or \emph{orthogonal} trajectories are obtained by
replacing $\Im$ by $\Re$ in the equation above. It is easy to check that
equation (\ref{re}) is equivalent to%
\[
\frac{\varphi_{a,b,\lambda}\left(  z\right)  }{\left(  z^{2}-1\right)  ^{2}%
}\,dz^{2}>0.
\]
The trajectories and the orthogonal trajectories of $\varpi_{a,b,\lambda}$
produce a transversal foliation of the Riemann sphere $\widehat{%
\mathbb{C}
}$. The only \textit{critical points} of $\varpi_{a,b,\lambda}$ are its zeros
(the roots $a$ and $b$ of $\varphi_{a,b,\lambda}$) and its poles, located at
$\pm1$ and at infinity; all others points of $%
\mathbb{C}
$ are regular.

A trajectory $\gamma$ of $\varpi_{a,b,\lambda}$ starting and ending at zeros
$a$ and $b$ (if exists) is called \emph{finite critical} or \emph{short}; if
it starts at one of the zeros $a$ or $b$ but tends to either pole, we call it
\emph{infinite critical trajectory} of $\varpi_{a,b,\lambda}$. In particular,
If $\gamma$ is a short trajectory joining $a$ and $b$, then, necessarily
\begin{equation}
\Im\int_{\gamma}\frac{\left(  \sqrt{\varphi_{a,b,\lambda}\left(  t\right)
}\right)  _{+}}{t^{2}-1}dt=0. \label{class homotopy}%
\end{equation}
Notice that any critical trajectory is either finite or infinite; any non
critical trajectory is either a loop, or it must diverge to infinite critical
points in its two directions.

The set of both finite and infinite critical trajectories of $\varpi
_{a,b,\lambda}$ together with their limit points (critical points of
$\varpi_{a,b,\lambda}$) is the \emph{critical graph} $\Gamma_{a,b}$ of
$\varpi_{a,b,\lambda}$. (See \cite{Jenkins}, or \cite{Striebel} for further
details on quadratic differentials)

In order to prove Theorem~\ref{thm:1} we start by analyzing the local
structure of the trajectories of $\varpi_{a,b,\lambda}$ at its critical points
(see e.g.~\cite{Jenkins},\cite{Striebel}). Recall that at any regular point
trajectories look locally as simple analytic arcs passing through this point,
and through every regular point of $\varpi_{a,b,\lambda}$ passes a uniquely
determined horizontal and uniquely determined vertical trajectory of
$\varpi_{a,b,\lambda}$, that are locally orthogonal at this point.
Furthermore, there are $3$ trajectories emanating from $a$ and from $b$ under
equal angles $2\pi/3$. The local structure of the trajectories near a double
pole depends on the vanishing of the real and imaginary parts of the residues
of the quadratic differential near this point.

Since
\begin{align*}
\varpi_{a,b,\lambda}  &  =\left(  \frac{\lambda^{2}\left(  1-a\right)  \left(
1-b\right)  }{\left(  z-1\right)  ^{2}}+\mathcal{O}\left(  \frac{1}%
{z-1}\right)  \right)  dz^{2},\quad z\rightarrow1,\\
\varpi_{a,b,\lambda}  &  =\left(  \frac{\lambda^{2}\left(  1+a\right)  \left(
1+b\right)  }{\left(  z+1\right)  ^{2}}+\mathcal{O}\left(  \frac{1}%
{z+1}\right)  \right)  dz^{2},\quad z\rightarrow-1,\\
\varpi_{a,b,\lambda}  &  =\left(  \frac{\lambda^{2}}{u^{2}}+\mathcal{O}\left(
\frac{1}{u^{3}}\right)  \right)  du^{2},\quad u\rightarrow0,\quad z=1/u,
\end{align*}
we conclude that the residues of $\varpi_{a,b,\lambda}$ at $-1,$ $1,$ and
$\infty$ are respectively $\lambda^{2}\left(  1-a\right)  \left(  1-b\right)
,$ $\lambda^{2}\left(  1+a\right)  \left(  1+b\right)  ,$ and $\lambda^{2}$.
Recall that the local behavior of the trajectories near a double pole has the
circle, the radial, or the log-spiral forms respectively if the residue there
is negative, positive, or non real. The existence of a short trajectory
joining $a$ or $b$ to itself implies that at list, one the residues above is negative.

Since $\varpi_{a,b,\lambda}$ has only three poles, Jenkins' three pole Theorem
asserts that it cannot have any recurrent trajectory.

\bigskip We denote $\mathcal{J}_{a,b}$ the set of all Jordan arcs joining $a$
and $b$ in $%
\mathbb{C}
\setminus\left\{  -1,1\right\}  $. Two arcs $\alpha,\beta:\left[  0,1\right]
\longrightarrow$ $%
\mathbb{C}
\setminus\left\{  -1,1\right\}  $ from $\mathcal{J}_{a,b}$ are homotopic if
there exists a continuous function $H:\left[  0,1\right]  \times$ $\left[
0,1\right]  \longrightarrow$ $%
\mathbb{C}
\setminus\left\{  -1,1\right\}  $ such that%
\[
\left\{
\begin{array}
[c]{c}%
H\left(  t,0\right)  =\alpha\left(  t\right)  \\
H\left(  t,1\right)  =\beta\left(  t\right)
\end{array}
,t\in\left[  0,1\right]  .\right.
\]
It is an equivalence relation on $\mathcal{J}_{a,b}$. It is well known that $%
\mathbb{C}
\setminus\left\{  -1,1\right\}  $ and the wedged two circles have the same
type of homotopy; in particular, there are four classes of equivalence of the
relation "homotopic" on $\mathcal{J}_{a,b}$, see Figure \ref{Fig3}.%
\begin{figure}[h]%
\centering
\fbox{\includegraphics[
height=2.0903in,
width=3.7048in
]%
{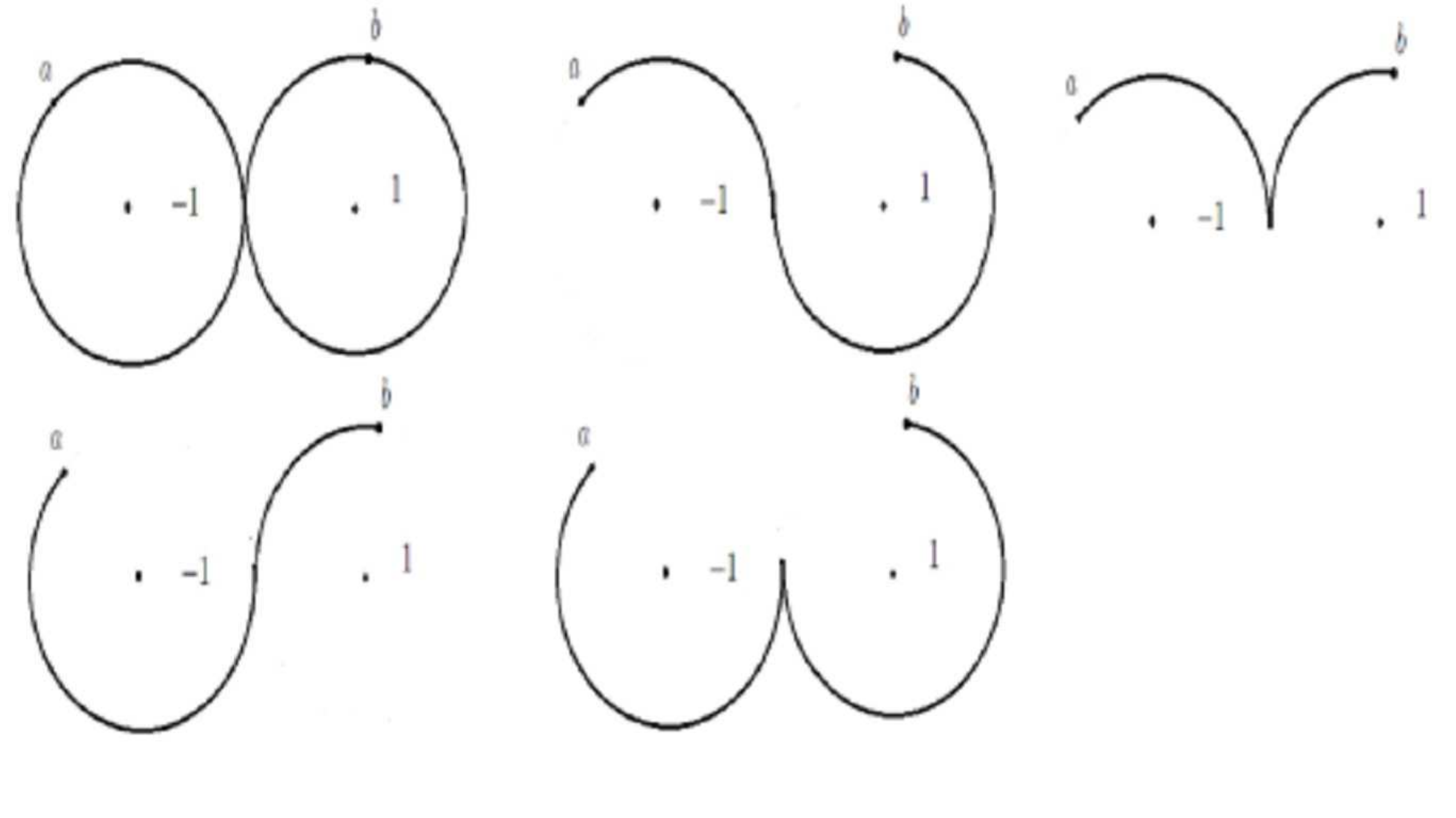}%
}\caption{Wedged circles and the 4 homotopic classes.}%
\label{Fig3}%
\end{figure}

For $\gamma_{0}\in\mathcal{J}_{a,b}$ fixed, we consider the single-valued
branch in $%
\mathbb{C}
\setminus\gamma_{0}$ of $\sqrt{\varphi_{a,b,\lambda}}$ fixed by the condition
\begin{equation}
\sqrt{\varphi_{a,b,\lambda}\left(  z\right)  }\sim\lambda z,z\rightarrow
\infty. \label{asymptInfity}%
\end{equation}
For $t\in\gamma_{0}\setminus\left\{  a,b\right\}  ,$ we denote by $\left(
\sqrt{\varphi_{a,b,\lambda}(t)}\right)  _{+}$ and $\left(  \sqrt
{\varphi_{a,b,\lambda}(t)}\right)  _{-}$ the limits from the $+$-side and
$-$-side respectively. (As usual, the $+$-side of an oriented curve lies to
the left, and the $-$-side lies to the right, if one traverses the curve
according to its orientation). Observe that%
\[
\left(  \sqrt{\varphi_{a,b,\lambda}\left(  t\right)  }\right)  _{+}=-\left(
\sqrt{\varphi_{a,b,\lambda}\left(  t\right)  }\right)  _{-},t\in\gamma
_{0}\setminus\left\{  a,b\right\}  .
\]
We have,

\begin{proposition}
\label{premier} Let $a,b,$ and $\lambda$ satisfy assumptions
(\ref{cond sur Aet B}), and let $\gamma_{0}$ be a Jordan arc in $%
\mathbb{C}
\setminus\{-1,1\}$ joining $a$ and $b$, and $\sqrt{\varphi_{a,b,\lambda}}$ is
its single-valued branch in $%
\mathbb{C}
\setminus\gamma_{0}$ fixed by the condition \ref{asymptInfity}. Then,
\[
\int_{\gamma_{0}}\frac{\left(  \sqrt{\varphi_{a,b,\lambda}\left(  t\right)
}\right)  _{+}}{t^{2}-1}dt=\pm i\pi\frac{\lambda}{2}\left(  \sqrt{\left(
1-a\right)  \left(  1-b\right)  }-\sqrt{\left(  1+a\right)  \left(
1+b\right)  }-2\right)  .
\]

\end{proposition}

\begin{proof}
Let $\gamma$ be a closed contour encircling the curve $\gamma_{0}$ and not
encircling $z=\pm1$. Thus,
\begin{align*}
\int_{\gamma_{0}}\frac{\left(  \sqrt{\varphi_{a,b,\lambda}\left(  t\right)
}\right)  _{+}}{t^{2}-1}dt  &  =\frac{1}{2}\int_{\gamma_{0}}\left[
\frac{\left(  \sqrt{\varphi_{a,b,\lambda}\left(  t\right)  }\right)  _{+}%
}{t^{2}-1}-\frac{\left(  \sqrt{\varphi_{a,b,\lambda}\left(  t\right)
}\right)  _{-}}{t^{2}-1}\right]  dt\\
&  =\frac{1}{2}\oint_{\gamma}\frac{\sqrt{\varphi_{a,b,\lambda}\left(
t\right)  }}{t^{2}-1}dt\\
&  =\pm i\pi\left(  \underset{-1}{res}+\underset{1}{res}+\underset{\infty
}{res}\right)  \left(  \frac{\sqrt{\varphi_{a,b,\lambda}\left(  t\right)  }%
}{t^{2}-1}\right) \\
&  =\pm i\pi\frac{\lambda}{2}\left(  \sqrt{\varphi_{a,b,\lambda}(1)}%
+\sqrt{\varphi_{a,b,\lambda}(-1)}-2\right) \\
&  =\pm i\pi\frac{\lambda}{2}\left(  \sqrt{\left(  1-a\right)  \left(
1-b\right)  }-\sqrt{\left(  1+a\right)  \left(  1+b\right)  }-2\right)  .
\end{align*}
The values of $\int_{\gamma}\frac{\sqrt{\varphi_{a,b,\lambda}\left(  t\right)
}}{t^{2}-1}dt$ for $\gamma\in\Gamma_{a,b}$ can be deduced by the knowledge of
the homotopic class of $\gamma,$ and applying Cauchy residue Theorem; without
consideration of the signs, these values are those numbers defined in
(\ref{im}).
\end{proof}

\bigskip\ As an immediate consequence is the following

\begin{lemma}
\label{necessaire}If the quadratic differential $\varpi_{a,b,\lambda}$ has a
short trajectory joining $a$ and $b,$ then, Property $\mathcal{P}%
_{a,b,\lambda}$ is satisfied$.$
\end{lemma}

The next tool we need to finish the proof of Theorem~\ref{thm:1} is the
so-called Teichm\"{u}ller lemma (see \cite[Theorem 14.1]{Striebel}), following
the idea already used in \cite{Atia},\cite{AMF FT}. Recall that a
$\varpi_{a,b,\lambda}$-polygon is any domain bounded only by trajectories or
orthogonal trajectories of $\varpi_{a,b,\lambda}$. If $z_{j}$ are its corners,
$n_{j}$ is the multiplicity of $z_{j}$ as a singularity of $\varpi
_{a,b,\lambda}$ (taking $n_{j}=1$ if $z_{j}\in\{a,b\}$, $n_{j}=0$ if it is a
regular point, and $n_{j}=-2$ if it is a double pole), and $\theta_{j}%
\in\left[  0,2\pi\right]  $ is the corresponding inner angle at $z_{j}$, then
\begin{equation}
\sum_{j}\beta_{j}=2+\sum_{i}n_{i},\quad\text{where }\beta_{j}=1-\theta
_{j}\frac{n_{j}+2}{2\pi}, \label{Teich}%
\end{equation}
and the summation in the right hand side goes along all zeroes of
$\varpi_{a,b,\lambda}$ inside the $\varpi_{a,b,\lambda}$-polygon. As an
immediate consequence, we have

\begin{lemma}
\label{2TRAJ HOM}There cannot exist two short trajectories that are homotopic
in the punctured plane $%
\mathbb{C}
\setminus\left\{  -1,1\right\}  .$
\end{lemma}

\begin{proof}
If such two short trajectories exist, then they will form an $\varpi
_{a,b,\lambda}$-polygon splitting $\widehat{%
\mathbb{C}
}$ into two connected domains; let $D$ be the bounded one. Clearly,
$D\cap\left\{  -1,1\right\}  =\emptyset,$ and then, the interior angles of
this $\varpi_{a,b,\lambda}$-polygon equal $\frac{2\pi}{3},$ therefore, the
left-hand side of (\ref{Teich}) equals $0$, whereas the right-hand is $2$, a contradiction.
\end{proof}

\begin{proposition}
\label{2traj1pole}Suppose that Property $\mathcal{P}_{a,b,\lambda}$ is
satisfied$.$ Then, there cannot exist two infinite critical trajectories
emanating from the same zero $a$ or $b$, and diverging to the same pole.
\end{proposition}

\begin{proof}
Assume that $\gamma^{-}$ and $\delta^{-}$ are two infinite critical
trajectories emanating from the same zero (for example $a$ ), spacing with
angle $\theta_{-}$, diverging to the same pole, for example, $z=-1$. We treat
the case when the residue $\lambda^{2}\left(  1+a\right)  \left(  1+b\right)
$ of the quadratic differential $\varpi_{a,b,\lambda}$ at the pole $-1$ is not
real. Let $\sigma$ be an orthogonal trajectory (not necessary critical)
diverging to $z=-1.$ Clearly, $\sigma$ intersects $\gamma^{-}$ and $\delta
^{-}$ alternatively infinitely many times; let $A$ and $B$ be two consecutive
intersections. Let $\Omega$ be an $\varpi_{a,b,\lambda}$-polygon $D,$ with
vertices $a,$ $A,$ and $B,$ and with edges, the arcs of $\gamma^{-}$,
$\delta^{-}$ and $\sigma$ connecting respectively $a$ and $A,$ $A$ and $B,$
and, $B$ and $a.$The interior angles of $\Omega$ at $A$ and $B$ are equal to
$\frac{\pi}{2}.$ Direct calculation shows that for $\Omega,$ the right hand
side of (\ref{Teich}) equals $1$ if $\theta_{-}=\frac{2\pi}{3},$ or $0$ if
$\theta_{-}=\frac{4\pi}{3};$ it follows that :

If $\theta_{-}=\frac{2\pi}{3}$, then $\Omega$ must contain $b$ and another
pole, necessarily, $z=1$. We conclude that the third trajectory emanating from
$a$ diverges to $\infty,$ and that all trajectories emanating from the other
zero $b$ stay inside $D;$ the same reasoning applied on trajectories emanating
from $b$ shows that two of them, say $\gamma_{1}^{+}$ and $\gamma_{2}^{+},$
diverge to $-1$ (with angle $\frac{4\pi}{3}$ at $b$ ), and the third one
diverges alone to $1.$ We may assume without loss of generalities, that
$\sigma$ intersects $\gamma^{-},$ $\gamma_{1}^{+},$ $\gamma_{2}^{+},$
$\delta^{-},$ and again $\gamma^{-},$ respectively in $A,$ $B,$ $C,$ $D,$ and
$E.$ We construct $4$ paths connecting $a$ and $b$ as follows (see Figure
\ref{Fig4}) :

\begin{itemize}
\item $\gamma_{0}:$ formed by the part of $\gamma_{1}^{+}$ joining $b$ to $B,$
the part of $\sigma$ joining $B$ to $A,$ and the part of $\gamma_{-}$ joining
$A$ to $a.$

\item $\gamma_{-1}:$ formed by the part of $\gamma_{1}^{+}$ joining $b$ to
$B,$ the part of $\sigma$ joining $B$ to $E,$ and the part of $\gamma_{-}$
joining $E$ to $a.$

\item $\gamma_{1}:$ formed by the part of $\gamma_{2}^{+}$ joining $b$ to $C,$
the part of $\sigma$ joining $C$ to $A,$ and the part of $\gamma_{-}$ joining
$A$ to $a.$

\item $\gamma_{\pm1}:$ formed by the part of $\gamma_{2}^{+}$ joining $b$ to
$C,$ the part of $\sigma$ joining $C$ to $E,$ and the part of $\gamma_{-}$
joining $A$ to $a.$
\end{itemize}

Clearly, these paths are not homotopic in $%
\mathbb{C}
\setminus\left\{  -1,1\right\}  ,$ and we have, by definition of trajectories
and orthogonal trajectories%
\begin{align*}
\Im\int_{\gamma_{0}}\frac{\left(  \sqrt{\varphi_{a,b,\lambda}\left(  t\right)
}\right)  _{+}}{t^{2}-1}dt &  =\Im\int_{b}^{B}\frac{\left(  \sqrt
{\varphi_{a,b,\lambda}\left(  t\right)  }\right)  _{+}}{t^{2}-1}dt\\
&  +\Im\int_{B}^{A}\frac{\left(  \sqrt{\varphi_{a,b,\lambda}\left(  t\right)
}\right)  _{+}}{t^{2}-1}dt\\
&  +\Im\int_{A}^{a}\frac{\left(  \sqrt{\varphi_{a,b,\lambda}\left(  t\right)
}\right)  _{+}}{t^{2}-1}dt\\
&  =\Im\int_{B}^{A}\frac{\left(  \sqrt{\varphi_{a,b,\lambda}\left(  t\right)
}\right)  _{+}}{t^{2}-1}dt\\
&  \neq0.
\end{align*}
With the same way, we prove that the imaginary part of the integral of
$\frac{\left(  \sqrt{\varphi_{a,b,\lambda}}\right)  _{+}\left(  z\right)
}{z^{2}-1}$ along each one of the paths $\gamma_{1},\gamma_{-1},$ and
$\gamma_{\pm1}$ cannot vanish, which contradicts Proposition \ref{premier}.%
\begin{figure}[h]%
\centering
\fbox{\includegraphics[
height=2.0903in,
width=3.7048in
]%
{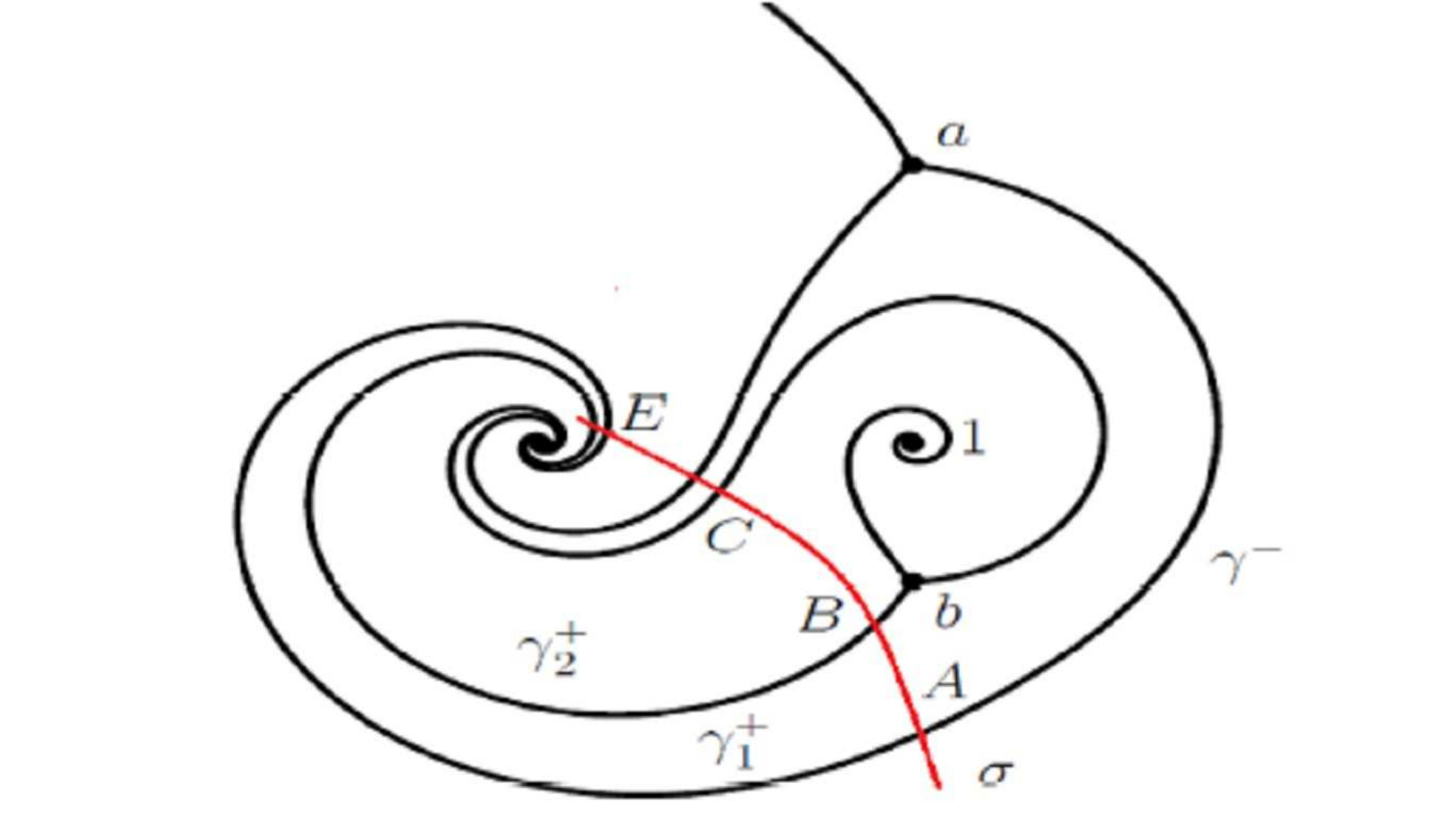}%
}\caption{Construction of 4 paths not homotopic in $\mathbb{C} \backslash
\left\{  -1,1\right\}  $. Here, $a=1-i$, $b=1+2i,$ and $\lambda=1+0.5i.$}%
\label{Fig4}%
\end{figure}

If $\theta_{-}=\frac{4\pi}{3},$ then the right hand side of \ref{Teich} equals
$0,$ and then, $\Omega$ must contain only the pole $1.$ We conclude that the
third trajectory emanating from $a$ diverges to $1$, and no trajectory
emanating from $b$ diverges to $1.$ The same reasoning applied on trajectories
emanating from $b$ shows that at list, one of them diverges to $-1.$ By
changing the roles of $b$ and $a$ in the previous case, we get the same contradiction.

The case $\lambda^{2}\left(  1+a\right)  \left(  1+b\right)  >0$ can be
treated in the same vein.
\end{proof}

\begin{proposition}
\label{2traj to each pole}If Property $\mathcal{P}_{a,b,\lambda}$ is
satisfied, then it cannot happen that to each pole, $+1,-1,$ and $\infty,$
there diverge two infinite critical trajectories.
\end{proposition}

\begin{proof}
If such a situation happens, then the critical graph $\Gamma_{a,b}$ splits $%
\mathbb{C}
$ into three connected domains of the strip types $\Omega_{0},\Omega_{1},$ and
$\Omega_{2};$ see Figure \ref{Fig5}. Let $\gamma_{0}$,$\gamma_{1},$ and
$\gamma_{-1}$ be three Jordan arc joining $a$ and $b$ respectively in the
domains $\Omega_{0},\Omega_{1},$ and $\Omega_{2}.$ We define the square root
$\sqrt{\varphi_{a,b,\lambda}}$ in $%
\mathbb{C}
\setminus\gamma_{0}$ with condition \ref{asymptInfity}. Clearly, the paths
$\gamma_{0}$,$\gamma_{1},$ and $\gamma_{-1}$ belong to three different classes
of homotopy in $%
\mathbb{C}
\setminus\left\{  -1,1\right\}  .$ Consider an orthogonal trajectory (not
necessary critical) $\sigma_{1}$ diverging to $z=1.$ With the same way of the
proof of Proposition \ref{2traj1pole}, we construct two paths connecting $a$
and $b,$ that are respectively homotopic in $%
\mathbb{C}
\setminus\left\{  -1,1\right\}  $ to $\gamma_{0}$ and $\gamma_{1},$ and we
get
\[
\Im\int_{\gamma_{0}}\frac{\left(  \sqrt{\varphi_{a,b,\lambda}\left(  t\right)
}\right)  _{+}}{t^{2}-1}dt\neq0,\text{ }\Im\int_{\gamma_{1}}\frac
{\sqrt{\varphi_{a,b,\lambda}\left(  t\right)  }}{t^{2}-1}dt\neq0.
\]
With another orthogonal trajectory $\sigma_{-1}$ diverging to $z=-1,$ we get
\[
\Im\int_{\gamma_{2}}\frac{\sqrt{\varphi_{a,b,\lambda}\left(  t\right)  }%
}{t^{2}-1}dt\neq0.
\]
Let us denote by $\gamma_{a}$ and $\gamma_{b}$ the critical trajectories that
emanate respectively from $a$ and $b,$ and diverge respectively to $-1$ and
$\infty.$ Let $\sigma$ be an orthogonal trajectory that diverges to $-1 $ and
$\infty.$ $\sigma$ intersects $\gamma_{a}$ and $\gamma_{b}$ infinitely many
times; let $A,B,$ and $C$ be respectively its first and second intersection
with $\gamma_{a},$ and its first intersection with $\gamma_{b}$. We consider
the paths $\gamma$ and $\gamma^{\prime}$

\begin{itemize}
\item $\gamma:$ formed by the part of $\gamma_{a}$ joining $a$ to $B,$ the
part of $\sigma$ joining $B$ to $A,$ and the part of $\gamma_{b}$ joining $A$
to $b.$

\item $\gamma^{\prime}:$ formed by the part of $\gamma_{a}$ joining $a$ to
$C,$ the part of $\sigma$ joining $C$ to $A,$ and the part of $\gamma_{b}$
joining $A$ to $b.$
\end{itemize}

Clearly, $\gamma$ and $\gamma^{\prime}$ are not homotopic in $%
\mathbb{C}
\setminus\left\{  -1,1\right\}  ,$ moreover, one of them, we denote it by
$\gamma_{3}$ is not homotopic in $%
\mathbb{C}
\setminus\left\{  -1,1\right\}  $ to $\gamma_{0},\gamma_{1},$ and $\gamma
_{2};$ besides, we have
\[
\Im\int_{\gamma_{3}}\frac{\sqrt{\varphi_{a,b,\lambda}\left(  t\right)  }%
}{t^{2}-1}dt\neq0.
\]
Finally, each one of the paths $\gamma_{0},\gamma_{1},\gamma_{2},$ and
$\gamma_{3}$ belong to different homotopy class in $%
\mathbb{C}
\setminus\left\{  -1,1\right\}  ,$ and Property $\mathcal{P}_{a,b,\lambda}$ is
violated.
\begin{figure}[h]%
\centering
\fbox{\includegraphics[
height=2.0903in,
width=3.7048in
]%
{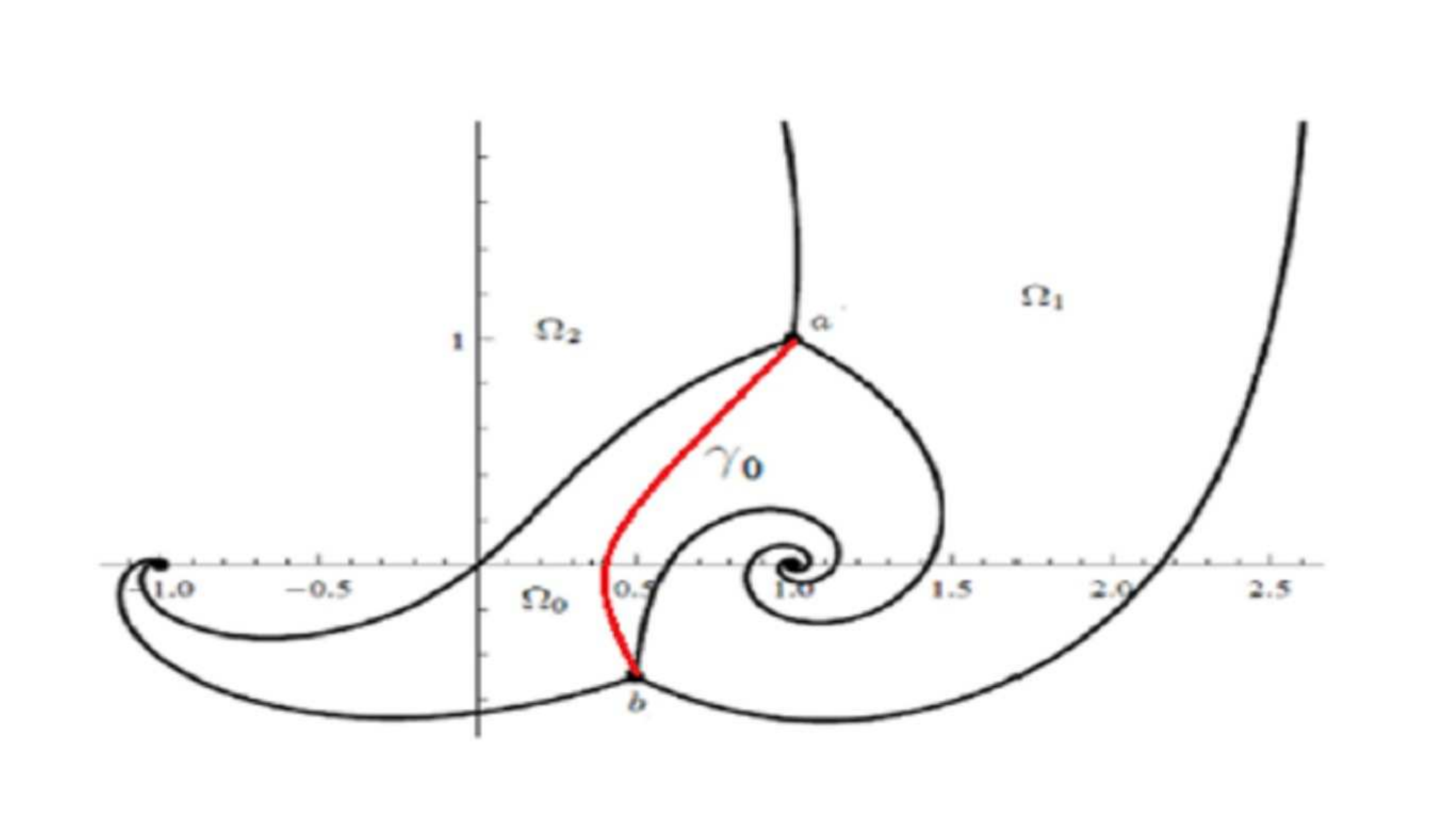}%
}\caption{The path $\gamma_{0}$ (red), and the domains $\Omega_{0},\Omega
_{1},$ and $\Omega_{2}$; here, $a=1+i,b=$ $0.5-0.5i$, and $\lambda=0.1+i.$}%
\label{Fig5}%
\end{figure}

\end{proof}

\begin{proof}
[Proof of Proposition \ref{main prop}]The necessary condition is done in Lemma
\ref{necessaire}.

From each zero $a$ and $b$ there emanate three critical trajectories; if
Property $\mathcal{P}_{a,b,\lambda}$ is satisfied, then, by Propositions
\ref{2traj1pole}, and \ref{2traj to each pole}, these six critical
trajectories cannot diverge all to the poles $-1,1,$ and $\infty$. We conclude
that necessarily we have at list one short trajectory. If Property
$\mathcal{P}_{a,b,\lambda}$ is satisfied for exactly one value from the four
possible values that can take, then, Lemma \ref{2TRAJ HOM} insure the
uniqueness of the short trajectory.
\end{proof}

\section{Proof of Theorem \ref{thm:1}}

\bigskip\ The zeros of $R_{A,B}$ defined in theorem \ref{shapiro+solynin} are
\begin{equation}
\zeta_{\pm}=\dfrac{-A^{2}+B^{2}\pm4\sqrt{\left(  A+1\right)  \left(
B+1\right)  \left(  A+B+1\right)  }}{\left(  A+B+2\right)  ^{2}},
\label{zeros}%
\end{equation}
respectively, in a way that
\[
R_{A,B}\left(  z\right)  =\left(  A+B+2\right)  ^{2}\left(  z-\zeta
_{+}\right)  \left(  z-\zeta_{-}\right)  .
\]
Since $R_{A,B}\left(  -1\right)  =4B^{2}$ and $R_{A,B}\left(  1\right)
=4A^{2},$ it is obvious that for $A$ and $B$ satisfying (\ref{cond sur Aet B}%
), $\ \zeta_{-}$ and $\zeta_{+}$ are simple and different from $\pm1$. The
quadratic differential $\varpi_{A,B}$ can be written $\varpi_{A,B}%
=\varpi_{\zeta_{-},\zeta_{+},\lambda}$ with $\lambda=i\left(  A+B+2\right)  .$
The residues of $\varpi_{\zeta_{-},\zeta_{+},\lambda}$ at the poles $1,-1,$
and $\infty$ are respectively $-A^{2},-B^{2},$ and $-(A+B+2)^{2};$ we conclude
that the trajectories have the radial or the log-spiral form in a neighborhood
of $\pm1$, respectively if $A,B\in i%
\mathbb{R}
$ or $A^{2},B^{2}$ $\notin%
\mathbb{R}
;$ and the circular, the radial, or the log-spiral form at $\infty,$
respectively if $\left(  A+B+2\right)  \in%
\mathbb{R}
,$ or $\left(  A+B+2\right)  \in i%
\mathbb{R}
,$ or $\left(  A+B+2\right)  \notin%
\mathbb{R}
$.

Straightforward calculations shows that%

\begin{align*}
\left(  A+B+2\right)  \sqrt{\left(  \zeta_{+}-1\right)  \left(  \zeta
_{-}-1\right)  }  &  =\allowbreak\pm2A,\\
\left(  A+B+2\right)  \sqrt{\left(  \zeta_{+}+1\right)  \left(  \zeta
_{-}+1\right)  }  &  =\allowbreak\pm2B,
\end{align*}
and we have then, from Proposition \ref{premier}

\begin{proposition}
\label{second}Let $A,B$ satisfy assumptions (\ref{cond sur Aet B}), and let
$\gamma$ be a Jordan arc in $%
\mathbb{C}
\setminus\{-1,1\}$ joining $\ \zeta_{-}$ and $\zeta_{+}$, and $\sqrt{R_{A,B}}$
is its single-valued branch in $%
\mathbb{C}
\setminus\gamma$ fixed by the condition \ref{asymptInfity}. Then
\[
\int_{\gamma}\frac{\left(  \sqrt{R_{A,B}\left(  t\right)  }\right)  _{+}%
}{t^{2}-1}dt\in\pm2\pi i\left\{  1,\left(  A+1\right)  ,\left(  B+1\right)
,\left(  A+B+1\right)  \right\}  .
\]

\end{proposition}

Taking into account assumptions (\ref{cond sur Aet B}), it follows from Lemma
(\ref{necessaire}) and Proposition \ref{second}, that, if $A+B\notin%
\mathbb{R}
,$ then there is always exactly one short trajectory of the quadratic
differential $\varpi_{A,B}$ connecting $\zeta_{-}$ and $\zeta_{+}.$

If $A+B\in%
\mathbb{R}
,$ then the critical graph is bounded; with the same idea of the proof of
Proposition \ref{2traj1pole}, one can show that there are at most two critical
trajectories diverging to the poles $\pm1.$ We conclude the existence of two
two short trajectories, the case of short trajectory connecting a zero to
itself can be discarded easily by Lemma \ref{Teich}. See Figure \ref{Fig6}.
\begin{figure}[h]%
\centering
\fbox{\includegraphics[
height=2.0903in,
width=3.7048in
]%
{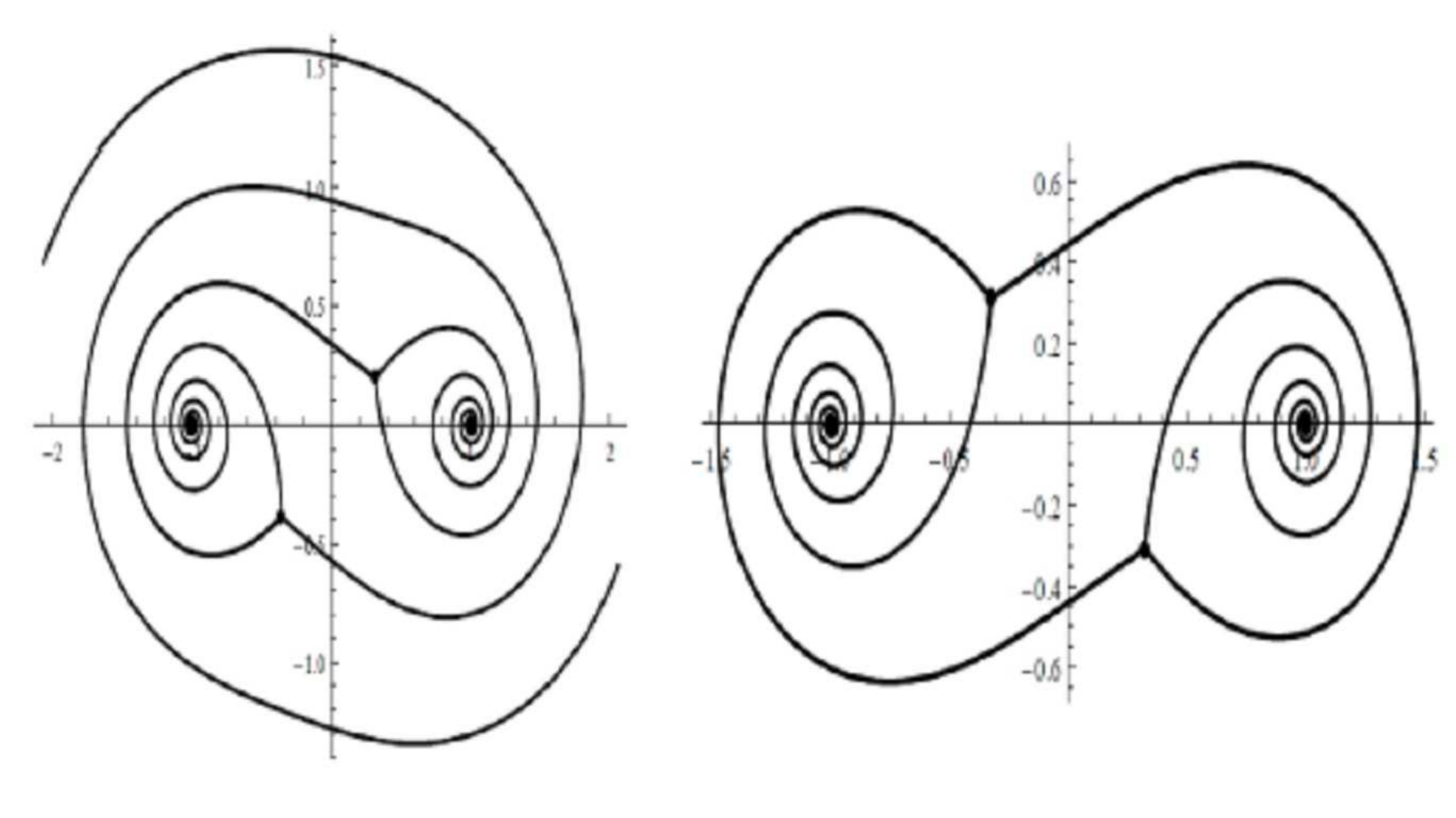}%
}\caption{Critical graphs when, $A+B\protect\notin\mathbb{R} ,$ here
$A=1+0.1i$ and $B=-1+0.1i$ (right); $A+B\in\mathbb{R} ,$ here $A=1+0.1i$ and
$B=-1-0.1i$ (left).}%
\label{Fig6}%
\end{figure}

\begin{acknowledgement}
The authors are grateful to Professor Said Zarati from the University of El
Manar, Tunis, for helpful discussions on Homotopy theory. This work was
partially supported by the research unit UR11ES87 from the University of
Gab\`{e}s in Tunisia.
\end{acknowledgement}

\texttt{Institut Sup\'{e}rieur des Sciences Appliqu\'{e}es et de
Technologie,6029 Avenue Omar Inb AL Khattab, Gab\`{e}s. Tunisia.}

\texttt{E-mail address:} \texttt{chouikhi.mondher@gmail.com}

\texttt{Institut Sup\'{e}rieur des Sciences Appliqu\'{e}es et de
Technologie,6029 Avenue Omar Inb AL Khattab, Gab\`{e}s. Tunisia. }

\texttt{E-mail address:} \texttt{faouzithabet@yahoo.fr}

$\ $

\begin{thebibliography}{99}                                                                                               %


\bibitem {Atia}M. J. Atia, Andrei Martinez-Finkelshtein, Pedro
Martinez-Gonzalez, and F. Thabet, Quadratic differentials and asymptotics of
Laguerre polynomials with varying complex parameters, J. Math. Anal. Appl. 416 (2014).

\bibitem {Atia2}M. J. Atia and F. Thabet, Quadratic differentials A$\left(
z-a\right)  \left(  z-b\right)  $dz$^{2}$/$\left(  z-c\right)  ^{2}$ and
algebraic Cauchy transform. arxiv1506.06543v1.To appear in C.J.M.

\bibitem {AMF FT}Andrei Martinez-Finkelshtein, Pedro Martinez-Gonzalez, and
Faouzi Thabet, Trajectories of quadratic differentials for Jacobi polynomials
with complex parameters. preprint submitted. Arxiv 1506.03434v1.

\bibitem {bullgard}T. Bergkvist and H. Rullg ard, On polynomial eigenfunctions
for a class of di erential oper-ators, Math. Res. Lett., vol.9, (2002), 153-171.

\bibitem {Gawronsky}Wolfgang Gawronski and Bruce Shawyer, Strong asymptotics
and the limit distribution of the zeros of Jacobi polynomials P(an+ ;bn+ )n
,Progress in approximation theory, Academic Press, Boston, MA, 1991,pp. 379.

\bibitem {Gonchar Rakhmanov}A. A. Gonchar and E. A. Rakhmanov, Equilibrium
measure and the distribution of zeros of extremal polynomials, Mat. Sbornik
125 (1984), no. 2, 117-127, translation from Mat. Sb., Nov. Ser. 134(176),
No.3(11),306-352 (1987).

\bibitem {Jenkins}James A. Jenkins, Univalent functions and conformal mapping,
Ergebnisse der Mathematik und ihrer Grenzgebiete. Neue Folge, Heft 18.Reihe:
Moderne Funktionentheorie, Springer-Verlag, Berlin, 1958.

\bibitem {ABJKuij AMF}A. B. J. Kuijlaars and A. Martinez-Finkelshtein, Strong
asymptotics for Jacobi polynomials with varying nonstandard parameters, J.
Anal. Math. 94 (2004), 195.

\bibitem {ABJK AMF RO}A. B. J. Kuijlaars, A. Martinez-Finkelshtein, and R.
Orive, Orthogonality of Jacobi polynomials with general parameters, Electron.
Trans. Numer.Anal. 19 (2005), 1.

\bibitem {AMF PGM RO}A. Martinez-Finkelshtein, P. Martinez-Gonzalez, and R.
Orive, Zeros of Jacobi polynomials with varying non-classical parameters,
Specialfunctions (Hong Kong, 1999), World Sci. Publ., River Edge, NJ, 2000,pp. 98.

\bibitem {AMF RO}A. Martinez-Finkelshtein and R. Orive, Riemann-Hilbert
analysis of Jacobi polynomials orthogonal on a single contour, J. Approx.
Theory 134 (2005), no. 2, 137.

\bibitem {AMF PGM RO Laguerre}Andrei Martinez-Finkelshtein, Pedro
Martinez-Gonzalez, and Ramon Orive, On asymptotic zero distribution of
Laguerre and generalized Bessel polynomials with varying parameters,
Proceedings of the Fifth International Symposium on Orthogonal Polynomials,
Special Functions and their Applications (Patras, 1999), vol. 133, 2001, pp. 477-487.

\bibitem {Deift}P. A. Deift, Orthogonal polynomials and random matrices: a
Riemann-Hilbert approach, New York University Courant Institute of
Mathematical Sciences, New York, 1999.

\bibitem {shapiro}Rikard B\o gvad and Boris Shapiro. On Motherbody Measures
with Algebraic Cauchy Transform. http://arxiv.org/abs/1406.1972.

\bibitem {shapiro soly}Boris Shapiro and Alexander Solynin, Root-counting
measures of Jacobi polynomials and topological types and critical geodisics of
related quadratic differentials.

\bibitem {Stahl}Herbert Stahl, Orthogonal polynomials with complex-valued
weight function. I, II, Constr. Approx. 2 (1986), no. 3, 225.

\bibitem {Striebel}Kurt Strebel, Quadratic differentials, Ergebnisse der
Mathematik und ihrer Grenzgebiete (3) [Results in Mathematics and Related
Areas (3)], vol. 5, Springer-Verlag, Berlin, 1984.

\bibitem {hyper}V. G. Bagrov and D. M. Gitman, Exact Solutions of Relativistic
Wave Equations,Kluwer Academic Publ., Dordrecht, 1990.

\bibitem {Szego}G. Szego, Orthogonal polynomials, fourth ed., Amer. Math. Soc.
Colloq. Publ., vol. 23, Amer. Math. Soc., Providence, RI, 1975.
\end{thebibliography}
\end{document}